\documentclass[12pt,english]{article}

\usepackage[margin = 2.5cm]{geometry,xcolor}

\usepackage{amsthm}
\usepackage{amsmath}
\usepackage{enumerate}
\usepackage{amssymb}
\usepackage{mathtools}
\usepackage{graphicx}
\usepackage[hidelinks, colorlinks, citecolor = blue, backref = page, bookmarks = false]{hyperref}
\usepackage{thm-restate}
\usepackage{enumitem}

\usepackage{caption}
\usepackage[square,sort,comma,numbers]{natbib} 
\usepackage{setspace}
\usepackage{cleveref}
\theoremstyle{plain}

\newtheorem*{theorem*}{Theorem}
\newtheorem{theorem}{Theorem}[section]
\crefname{theorem}{Theorem}{Theorems}
\Crefname{theorem}{Theorem}{Theorems}

\newtheorem*{lemma*}{Lemma}
\newtheorem{lemma}[theorem]{Lemma}
\crefname{lemma}{Lemma}{Lemmas}
\Crefname{lemma}{Lemma}{Lemmas}

\newtheorem*{claim*}{Claim}
\newtheorem{claim}[theorem]{Claim}
\crefname{claim}{Claim}{Claims}
\Crefname{claim}{Claim}{Claims}

\crefname{proposition}{Proposition}{Propositions}
\Crefname{proposition}{Proposition}{Propositions}

\newtheorem{corollary}[theorem]{Corollary}
\crefname{corollary}{corollary}{Corollaries}
\Crefname{corollary}{Corollary}{Corollaries}

\newtheorem{conjecture}[theorem]{Conjecture}
\crefname{conjecture}{Conjecture}{Conjectures}
\Crefname{conjecture}{Conjecture}{Conjectures}

\newtheorem{question}[theorem]{Question}
\crefname{question}{Question}{Questions}
\Crefname{question}{Question}{Questions}

\newtheorem{observation}[theorem]{Observation}
\crefname{observation}{Observation}{Observations}
\Crefname{observation}{Observation}{Observations}

\crefname{example}{Example}{Examples}
\Crefname{example}{Example}{Examples}

\theoremstyle{definition}

\crefname{problem}{Problem}{Problems}
\Crefname{problem}{Problem}{Problems}

\newtheorem{definition}[theorem]{Definition}
\crefname{definition}{Definition}{Definitions}
\Crefname{definition}{Definition}{Definitions}

\theoremstyle{remark}

\crefname{remark}{Remark}{Remarks}
\Crefname{remark}{Remark}{Remarks}

\usepackage{xpatch}
\xpatchcmd{\proof}{\itshape}{\normalfont\proofnamefont}{}{}
\newcommand{\proofnamefont}{}
\renewcommand{\proofnamefont}{\bfseries}

\usepackage{framed}

\usepackage{bm}

\newcommand{\remove}[1]{}

\newcommand{\ceil}[1]{
    \lceil #1 \rceil
}

\renewcommand{\angle}[1]{
    \langle #1 \rangle
}

\newcommand{\msubseteq}{\:\widetilde{\subseteq}\:}
\newcommand{\mcup}{\,\widetilde{\cup}\,}
\newcommand{\bigmcup}{\widetilde{\bigcup}}
\newcommand{\mcap}{\,\widetilde{\cap}\,}
\newcommand{\mminus}{-}

\newcommand{\G}{\Gamma}
\newcommand{\cG}{\mathcal{G}}

\newcommand{\cI}{\mathcal{I}}
\newcommand{\cX}{\mathcal{X}}
\newcommand{\cY}{\mathcal{Y}}
\newcommand{\Z}{\mathbb{Z}}
\newcommand{\F}{\mathbb{F}}

\DeclareMathOperator{\spa}{span}
\DeclareMathOperator{\rank}{rank}

\title{Directed cycles with zero weight in $\mathbb{Z}_p^k$}

\author{
	    Shoham Letzter\thanks{
		Department of Mathematics, 
		University College London, 
		Gower Street, London WC1E~6BT, UK. 
		Email: \texttt{s.letzter}@\texttt{ucl.ac.uk}. 
		Research supported by the Royal Society.
    }
    \and 
        Natasha Morrison\thanks{Department of Mathematics and Mathematical Statistics, University of Victoria, 3800 Finnerty Road, Victoria, BC V8P 5C2, Canada.
        Email: \texttt{nmorrison}@\texttt{uvic.ca}.
        Research supported by NSERC Discovery Grant RGPIN-2021-02511.}
}

\begin{document}

\date{\today}
\maketitle

\begin{abstract}
	\setlength{\parskip}{\medskipamount}
    \setlength{\parindent}{0pt}
    \noindent

	For a finite abelian group $A$,  define $f(A)$ to be the minimum integer such that for every complete digraph $\Gamma$ on $f$ vertices and every map $w:E(\Gamma) \rightarrow A$, there exists a directed cycle $C$ in $\Gamma$ such that $\sum_{e \in E(C)}w(e) = 0$. The study of $f(A)$ was initiated by Alon and Krivelevich (2021). In this article, we prove that $f(\Z_p^k) = O(pk (\log k)^2)$, where $p$ is prime, with an improved bound of $O(k \log k)$ when $p = 2$. These bounds are tight up to a factor which is polylogarithmic in $k$. 

\end{abstract}
\section{Introduction} \label{sec:intro}

	Generally speaking, the area of `zero-sum' Ramsey theory concerns the study of objects weighted by the elements of a group, and when a substructure of total weight zero can be found. This dates back to 1960 and the celebrated Erd\H{o}s--Ginzburg--Ziv~\cite{egz} theorem which asserts that, for every integer $m \ge 2$, every sequence of $2m-1$ elements in $\Z_m$ contains a subsequence of $m$ elements that sum to $0$. Since then, a wide variety of interesting variations have been studied; see Caro \cite{caro1996zero} for a survey of this topic. 

	For a finite abelian group $A$,  define $f(A)$ to be the minimum integer such that for every complete digraph $\Gamma$ on $f$ vertices and every map $w:E(\Gamma) \rightarrow A$, there exists a directed cycle $C$ in $\Gamma$ such that $\sum_{e \in E(C)}w(e) = 0$. We call $w$ a \emph{weighting} of $E(\Gamma)$ and $C$ a $w$-\emph{zero-sum cycle}. When it is clear which map $w$ is being referred to, we will simply call $C$ a \emph{zero-sum cycle}. The question of determining $f(A)$ arose in a paper by Alon and Krivelevich \cite{alon2020divisible} from 2021, who showed that for any integer $q$, if $f = f(\Z_q)$, then every $K_{2f}$-minor contains a cycle whose length is divisible by $q$ (see the end of Section 3 in \cite{alon2020divisible}). In fact, they proved more generally that for every subcubic graph $H$ and positive integer $q$, there is a (finite) number $g = g(H, q)$ such that every $K_g$-minor contains a subdivision of $H$ where each edge is replaced by a path of length which is divisible by $q$. Their bound on $g(H,q)$ was improved by Das, Dragani\'c, and Steiner \cite{das2021tight} who showed that $g(H, q) = O(|H|q)$, which is tight up to a constant factor.

	Alon and Krivelevich~\cite{alon2020divisible} proved that  $f(\Z_p) \le 2p-1$, for $p$ prime, and $f(\Z_q) = O(q \log q)$, for any integer $q \ge 2$. This result was improved upon and generalised by M\'{e}sz\'{a}ros and Steiner~\cite{MS21}, who showed that $f(A) \le 8|A|$ for any finite abelian group $A$, and in particular, that $f(\mathbb{Z}_p) \le \frac{3}{2}p$ for prime $p$. Recently, this was improved upon by Berendsohn, Boyadzhiyska, and Kozma~\cite{BBK22}, and independently by Akrami, Chaudhury, Garg, Mehlhorn, and Mehta \cite{akrami2022efx}, who gave a beautifully slick proof to show that $f(B) \le 2|B| - 1$, where $B$ is any finite (not necessarily abelian) group. In forthcoming work of Campbell, Hendrey, Gollin, and Steiner, this is improved to a tight bound $f(\Z_q) = q+1$ for every positive integer $q$.

    Our main result improves upon the results in \cite{MS21,BBK22,akrami2022efx} to give a sublinear bound in $|A|$ when $A = \mathbb{Z}_p^k$ and $p$ is a prime. 
	\begin{restatable}{theorem}{main}
		\label{thm:main}
		Let $p$ be a prime and let $k \ge 1$ be an integer. Then $$f(\Z_p^k) \le 600p \cdot k(\log_2 (10k))^2.$$
	\end{restatable}

	We obtain a stronger result when $p = 2$, which is tight up to an $O(\log k)$ factor. 
	\begin{restatable}{theorem}{pistwo}
		\label{thm:pis2}
		Let $k \ge 2$ be an integer. Then
		$$f(\Z_2^k) \le 600 k \log_2(2k).$$
	\end{restatable}

	A simple construction shows that $f(\Z_p^k) \ge (p-1)k$. Indeed, let $e_1, \ldots, e_k$ be the elementary basis elements of $\Z_p^k$. Consider the complete digraph $\Gamma$ on $(p-1)k$ vertices, let $\{V_1, \ldots, V_k\}$ be an equipartition of $V(\Gamma)$, and label a directed edge $xy$ by $e_i$ whenever $x \in V_i$. It is easy to see that, with this weighting, there are no zero-sum cycles in $\Gamma$, as every cycle contains at most $p-1$ edges labelled $e_i$, for every $i \in [k]$. Thus $f(\Z_p^k) \ge (p-1)k$, as claimed.
	The results in \Cref{thm:main} and \Cref{thm:pis2} are thus tight up to a factor which is polylogarithmic in $k$.

    This problem was independently investigated by Sidorenko and Steiner (unpublished) who showed $f(\Z_p^k) = O(pk^2 \log k)$ with an improved bound of $O(k^2)$ when $p = 2$.

	The proofs of Theorems~\ref{thm:main} and \ref{thm:pis2} follow the same strategy. In \Cref{lem:step} below, we show that $f(\Z_p^k)$ can be bounded by an expression involving the largest size of a `reduced' multisubset of $\Z_p^k$. This parameter is bounded in general for $\Z_p^k$ in  \Cref{thm:reduced-upper}, and a better bound is easily obtained when $p=2$, as we will see in \Cref{obs:reduced}. 

	In \Cref{sec:reduced} we introduce the notion of reduced sets, make some preliminary observations and prove \Cref{thm:reduced-upper}. Then, in \Cref{sec:gad}, we show how we can use particular subgraphs, called `gadgets', collectively to find zero-sum cycles. These allow us to deduce information about the weights on our digraph and, in \Cref{sec:proof}, to prove \Cref{lem:step}. This in turn completes the proofs of \Cref{thm:main} and \Cref{thm:pis2}. We conclude in \Cref{sec:conc} with some discussion about bounding the largest size of a `reduced' multisubset of $\Z_p^k$ and other directions for future research.

	\subsection{Notation and conventions}

		All logarithms will be taken in base $2$.

		For sets $A$ and $B$, we write $A + B := \{a + b: a \in A, b \in B\}$. 

		We will often treat collections $S$ of elements of $\Z_p^k$ as collections of vectors in $\F_p^k$. As such, we will refer to the span or dimension of such a collection $S$. We will use 0 to denote the identity element of any group $A$.

		We will often be working with multisets. In order to minimise confusion for the reader, we now define some notation for multiset operations that will be used throughout.

		Given an abelian group $A$ and a multiset $S$, write $S \msubseteq A$ to denote that $S$ is a multiset with the property that $s \in A$ for every $s \in S$. Write $T \subseteq S$ to denote that $T$ is a (multi)subset of $S$. We will write $S - T$ to denote the multiset obtained from $S$ by removing one copy of each element of the multiset $T$. 

		Define the \emph{multiset union} of sets $A_1,\ldots, A_k$, denoted by $\bigmcup_{i \in [k]} A_i$ to be the multiset whose elements can be partitioned into $A_1,\ldots,A_k$. If $a_1,\ldots, a_k$ are elements of $A$, write $\bigmcup_{i \in [k]}a_i$ to denote the multiset containing $\{a_1,\ldots,a_k\}$.

\section{Reduced sets}\label{sec:reduced}

 Given a multiset $S$ with elements in an abelian group $A$, the \emph{sumset} of $S$, denoted $\Sigma(S)$, is defined to be the set of all subset sums of $S$, namely
\begin{equation*}
    \Sigma(S) := \left\{\sum_{t \in T} t :  T \subseteq S\right\},
\end{equation*}
where the sum of elements in the empty set is defined to be zero, and hence $|\Sigma(\emptyset)| = 1$. In particular, $0 \in \Sigma(S)$ for every multiset $S \msubseteq A$. Note that, when working over $\mathbb{F}_2$, the sumset of $S$ is equal to the span of $S$. 
Say that $S$ is \emph{reduced} if $|\Sigma(S)| > |\Sigma(S \mminus \{s\})|$, for all $s \in S$. Let $h_p(k)$ be the size of a largest reduced multiset in $\Z_p^k$. 

The main result in this section is a general upper bound on $h_p(k)$. 

\begin{restatable}{theorem}{thmReducedUpper} \label{thm:reduced-upper}
	$h_p(k) \le (p-1)(\log k + 1) \cdot k$.
\end{restatable}
This shows that the lower bound $h_p(k) \ge k(p-1)$, from \Cref{obs:reduced}~\ref{itm:reduced-lower-bound} below, is tight up to a $\log k + 1$ factor. 

We begin by gathering together some straightforward observations about reduced multisets. 

\begin{observation} \label{obs:reduced} Let $A$ be a finite abelian group and let $S \msubseteq A$.
	\hfill
	\begin{enumerate}[label = \rm(\roman*)]
	    \item \label{itm:identity} If $S$ is reduced, then the identity element $0 \in A$ is not an element of $S$.
		\item \label{itm:cont-reduced}
	There exists $S' \subseteq S$ such that $S'$ is reduced and $\Sigma(S) = \Sigma(S')$. 
		\item \label{itm:reduced-subset}
			If $S$ is reduced, then every $T \subseteq S$ is reduced.
		\item \label{itm:reduced-map}
			Suppose that $S$ is reduced and $A = \Z_p^k$. Let $f : \Z_p^k \to \Z_p^k$ be an injective linear map. Then $f(S) := \{f(s) : s \in S\}$ is reduced.
		\item \label{itm:reduced-sub-additive}
			$h_p(k) + h_p(\ell) \le h_p(k+\ell)$.
		\item \label{itm:reduced-dim-1}
			$h_p(1) = p-1$.
		\item \label{itm:reduced-p-2}
			A multisubset of $\Z_2^k$ is reduced if and only if it is linearly independent. In particular, $h_2(k) = k$.
		\item \label{itm:reduced-lower-bound}
			$h_p(k) \ge k(p-1)$.
	\end{enumerate}
\end{observation}

\begin{proof}
    For \ref{itm:identity}, observe that $\Sigma(S) = \Sigma(S \mminus \{0\})$.  

    For \ref{itm:cont-reduced}, if $S$ is not reduced then sequentially remove elements $t$ such that $\Sigma(S \mminus \{t\}) = \Sigma(S)$ until a reduced multisubset $S'$ remains. 
    
	For \ref{itm:reduced-subset}, if $T \subseteq S$ and $\Sigma(T) = \Sigma(T \mminus \{t\})$ for some $t \in T$, then $\Sigma(S \mminus \{t\}) = \Sigma(S \mminus T) + \Sigma(T \mminus \{t\}) = \Sigma(S \mminus T) + \Sigma(T) = \Sigma(S)$, contradicting the assumption that $S$ is reduced. Thus $|\Sigma(T)| > |\Sigma(T \mminus \{t\})|$ for every $t \in T$, showing that $T$ is reduced.

	Item \ref{itm:reduced-map} follows from the observation that $\Sigma(f(T)) = f(\Sigma(T))$ for every multisubset $T$ of $\Z_p^k$.

	For \ref{itm:reduced-sub-additive}, let $S \msubseteq \Z_p^k$ and $T \msubseteq \Z_p^{\ell}$ be reduced with respective sizes of $h_p(k)$ and $h_p(\ell)$.
    Define $R := \{(s, 0) : s \in S\} \cup \{(0, t) : t \in T\} \subseteq \Z_k^p \times \Z_p^{\ell}$. It is easy to see that $R$ is a reduced multiset in $\Z_p^{k} \times \Z_p^{\ell} \cong \Z_p^{k+\ell}$ of size $h_p(k) + h_p(\ell)$, showing $h_p(k+\ell) \ge h_p(k) + h_p(\ell)$.

	For \ref{itm:reduced-dim-1}, suppose that $S \msubseteq \Z_p$ is reduced. Enumerate the elements of $S$ as $\{s_1, \ldots, s_k\}$. By \ref{itm:reduced-subset}, $\{s_1, \ldots, s_i\}$ is reduced for all $i \in [k]$, and so
	\begin{equation*}
		|\Sigma(S)| = \big|\Sigma(\{s_1, \ldots, s_k\})\big| 
		\ge \big|\Sigma(\{s_1, \ldots, s_{k-1}\})\big| + 1 
		\ge \ldots \ge |\Sigma(\emptyset)| + k = k + 1.
	\end{equation*}
	Since, trivially, $|\Sigma(S)| \le p$, we have $|S| = k \le p-1$.

	Thinking of $\Z_2^k$ as a vector space over $\F_2$, the sumset of a set $S \subseteq \Z_2^k$ is the same as the span of $S$.  Item \ref{itm:reduced-p-2} follows.

	Finally, for \ref{itm:reduced-lower-bound}, notice that the multiset consisting of $p-1$ copies of each elementary basis element of $\Z_p^k$ is a reduced multiset of size $k(p-1)$.
\end{proof}

\subsection{Proof of \Cref{thm:reduced-upper}}

	Before giving the proof, we will state two preliminary results. The first is the Matroid Packing theorem, due to Edmonds \cite{edmonds1965lehman}. Before stating it, we will define matroids and some relevant notions. 

	A \emph{matroid} is a pair $(V, \cI)$, where $V$ is a set and $\cI$ is a non-empty family of subsets of $V$, referred to as \emph{independent sets}, which is closed under taking subsets, and satisfies the following \emph{augmentation property}: if $I, I' \in \cI$ satisfy $|I| > |I'|$, then there exists $x \in I \setminus I'$ such that $I' \cup \{x\} \in \cI$.
	Define the \emph{rank} of a subset $S \subseteq V$, denoted $\rank(S)$, as the size of a largest independent set contained in $S$. A \emph{base} is an independent set of size $\rank(V)$.

	\begin{theorem}[Matroid Packing theorem, Edmonds \cite{edmonds1965lehman}] \label{thm:matroid}
		A matroid $M = (V, \cI)$ contains $t$ pairwise disjoint bases if and only if every $T \subseteq V$ satisfies $|V| - |T| \ge t \cdot (\rank(V) - \rank(T))$.
	\end{theorem}

	We draw the following almost immediate conclusion regarding disjoint bases in multisubsets of vector spaces. Given a (multi)subset $S$ of a vector space, define $\rank(S) := \dim(\spa(S))$.
	\begin{corollary} \label{cor:matroid}
		Let $U$ be a finite dimensional vector space and let $S \msubseteq U$ have full rank.
		Then $S$ contains $t$ pairwise disjoint bases of $U$ if and only if every multisubset $T \subseteq S$ satisfies $|S| - |T| \ge t \cdot (\rank(S) - \rank(T))$.
	\end{corollary}

	\begin{proof}
		We define a matroid $M = (V, \cI)$, as follows. For each element $s \in S$ of multiplicity $m_s$, add the elements $(s,1),\ldots, (s,m_s)$ to $V$ and say that each of these elements \emph{corresponds} to $s$. So $V$ is a set (rather than a multiset) which emulates $S$. Define $\cI$  to be the subsets of $V$ with the property that the collection of vectors they correspond to in $S$ is linearly independent. 
		It is easy to check that $M$ is indeed a matroid and that a base in $M$ corresponds to a basis of $U$.
		It follows from \Cref{thm:matroid} that $M$ contains $t$ pairwise disjoint bases of $U$ if and only if $|V| - |T| \ge t \cdot (\rank(V) - \rank(T))$ for every $T \subseteq V$. This proves the statement, as $|V| = |S|$ and $\rank(V) = \rank(S)$.
	\end{proof}	

We will also need the following result of Alon, Linial, and Meshulam~\cite{alon1991additive}.

	\begin{theorem}[Proposition 3.1 in \cite{alon1991additive}] \label{thm:additive-base}
		Let $S_1, \ldots, S_{\ell}$ be $\ell$ bases of the vector space $\Z_p^k$, where $\ell \ge (p-1)\log k + p-2$, and let $S = \bigmcup_{i \in \ell}S_i$. Then $\Sigma(S) = \Z_p^k$.
	\end{theorem}

	We now have all the pieces we need to prove \Cref{thm:reduced-upper}, restated here.
	\thmReducedUpper*

	\begin{proof}[Proof of \Cref{thm:reduced-upper}]
		We will prove the statement by induction on $k$. Notice that it follows directly from \Cref{obs:reduced}~\ref{itm:reduced-dim-1} if $k = 1$, so take $k \ge 2$ and assume that the statement holds for all $k' \in [k-1]$. 
		Write $\ell_k = (p-1)(\log k + 1)$. Suppose that $T \msubseteq \mathbb{Z}_p^k$ is reduced and $|T| > \ell_k \cdot k$. Let $S \subseteq T$ such that $|S| = \ell_k \cdot k$. By \Cref{obs:reduced}~\ref{itm:reduced-subset}, $S$ is reduced. We will show that $\Sigma(S) = \Z_p^k$, contradicting the assumption that $T$ is reduced.  
		
		By induction, every $S' \msubseteq S$ with $r:= \rank(S') < k$ satisfies $|S'| \le \ell_{r} \cdot r \le \ell_k \cdot \rank(S')$, because $S'$ is a reduced multisubset of an $r$-dimensional subspace of $\Z_p^k$ (using \Cref{obs:reduced}~\ref{itm:reduced-subset} and~\ref{itm:reduced-map}). In particular, we have $\rank(S) = k$, and if $S'\msubseteq S$ satisfies $\rank(S') < k$, then
		\begin{equation*}
			\frac{|S| - |S'|}{\rank(S) - \rank(S')}
			\ge \frac{\ell_k \cdot k - \ell_k \cdot \rank(S')}{k - \rank(S')} \ge \ell_k.
		\end{equation*}
		It thus follows from \Cref{cor:matroid} that $S$ is the multiset union of $\ell_k$ pairwise disjoint bases. By \Cref{thm:additive-base}, this implies that $\Sigma(S) = \Z_p^k$, as suffices to complete the result. \end{proof}

\section{Gadgets}\label{sec:gad}
	Let $\Gamma$ be a complete digraph, $A$ be an abelian group, and let $w: E(\Gamma) \rightarrow A$. We will write $uv$ to denote the directed edge from $u$ to $v$.

	\begin{definition}[Gadgets]
		Let $u$ and $v$ be distinct vertices in $\Gamma$. A \emph{gadget} rooted at $(u,v)$ is defined to be a pair $g$ of paths $P$ and $Q$ directed from $u$ to $v$. Let $u(g), v(g), P(g)$ and $Q(g)$ refer to $u, v, P$ and $Q$, respectively.

		Define the \emph{vertex set} of $g$ by $V(g):= V(P) \cup V(Q)$, and its \emph{value} by $g^* := w(Q) - w(P)$, where the \emph{weight} of a subdigraph $H \subseteq \Gamma$ is $w(H) := \sum_{e \in E(H)} w(H)$.

	\end{definition}

	The next lemma provides a simple condition for a collection of gadgets to yield a zero-sum cycle in $\Gamma$. The point is that if the collection of values is `rich' enough, then by passing through the gadgets, choosing either $P$ or $Q$ at each one, we are able to use the collection to generate a zero-sum cycle. 

	\begin{lemma}\label{lem:sigmall}
		Let $\Gamma$ be a complete digraph whose edges have weights in the abelian group $A$.
		Let $\cG$ be a family of pairwise vertex-disjoint gadgets in $\G$. Let $S := \bigmcup_{ g \in \cG} g^*$. If $\Sigma(S) = A$, then $\G$ contains a zero-sum cycle. 
	\end{lemma}
	\begin{proof}
		Let $\cG = \{g_1,\ldots, g_t\}$, where $g_i = g_i(u_i,v_i)$. For ease of notation, define $u_{t+1}:= u_1$. Let $I \subseteq [t]$ be a set of indices such that the following holds, where $(u_i, v_i, P_i, Q_i) := (u(g_i), v(g_i), P(g_i), Q(g_i))$.
		\begin{equation*}
			\sum_{i \in I}g^*_i 
			= - \sum_{i = 1}^t w(P_i) - \sum_{i=1}^{t} w(v_i u_{i+1}),
		\end{equation*}
		Such a set exists as $\Sigma(S) = A$.
		Using $\sum_{i \in I} g_i^* = \sum_{i \in I} \big(w(Q_i) - w(P_i)\big)$ and rearranging yields
		\begin{equation}\label{eq:Cwt}
			0 = \sum_{i \in I}w(Q_i) + \sum_{i \notin I}w(P_i) + \sum_{i=1}^{t} w(v_iu_{i+1}).
		\end{equation}

		Now let $C$ be the directed cycle $R_1v_1u_2R_2v_2u_2 \ldots R_t v_t u_1$, where 
		$$R_i:= \begin{cases}
			Q_i & i \in I\\
			P_i & i \notin I.
		\end{cases}$$
		Note that $\sum_{e \in C}w(e)$ is precisely the right hand side of the expression \eqref{eq:Cwt}, and so $C$ is a zero-sum cycle. 
	\end{proof}

	The following lemma allows us to modify a weighting to ensure all out-edges from a particular vertex receive weight 0, whilst preserving properties of the weighting. The idea of locally modifying weights of edges adjacent to a fixed vertex without creating or destroying zero-sum cycles was also used by M\'esz\'aros and Steiner~\cite{MS21}. 

	\begin{lemma} \label{lem:zero-vertex}
		Let $\Gamma$ be a complete digraph, let $A$ be an abelian group, and let $w : E(\G) \to A$. Suppose that there are no zero-sum cycles with respect to $w$, and let $v_0 \in V(\G)$. Then there exists $w' : E(\G) \to A$ such that: $w'(v_0 u) = 0$ for every $u \in V(\G)$; there are no zero-sum cycles with respect to $w'$; and every gadget $g$ in $\G$ has the same value with respect to both $w$ and $w'$.
	\end{lemma}

	\begin{proof}
		Let $u \in V(\G)$ and let $\alpha \in A$. Define $f_{u, \alpha}(w) : E(\Gamma) \to A$ as follows
		\begin{equation*}
			f_{u, \alpha}(w)(e) := \left\{
				\begin{array}{ll}
					w(e) & \text{ if $e \in \Gamma \setminus \{u\}$} \\
					w(e) + \alpha & \text{ if $e = uv$ for some $v \in V(\G)$} \\
					w(e) - \alpha & \text{ if $e = vu$ for some $v \in V(\G)$.}
				\end{array}
				\right.
			\end{equation*}
			It is easy to check that $w(C) = f_{u, \alpha}(w)(C)$ for every directed cycle $C$. Similarly, if $P$ and $Q$ are two directed paths with the same start and end points, then $w(P) - w(Q) = f_{u, \alpha}(w)(P) - f_{u, \alpha}(w)(Q)$. 
			It follows that there are no $w$-zero-sum cycles if and only if there are no $f_{u, \alpha}(w)$-zero-sum cycles, and that the value of any gadget $g$ is the same with respect to $w$ and $f_{u, \alpha}(w)$.

			Let $u_1, \ldots, u_{n-1}$ be an arbitrary ordering of $V(\G) \setminus \{v_0\}$, and let $\alpha_i := w(v_0 u_i)$. Let 
			\begin{equation*}
				w' := f_{u_{n-1}, \alpha_{n-1}} \circ \ldots \circ f_{u_1, \alpha_1} (w).
			\end{equation*}
			Notice that $w'(v_0 u_i) = w(v_0 u_i) - \alpha_i = 0$, for every $i \in [n-1]$; equivalently, $w'(v_0 u) = 0$ for every $u \in V(\G) \setminus \{v_0\}$. By the above discussion, there are no $w'$-zero-sum cycles, and every gadget in $\G$ has the same value with respect to $w'$ and $w$.
		\end{proof}

		We now introduce some notation and terminology that we will use to refer to particular families of gadgets. 

		\begin{definition}
			Let $\Gamma$ be a complete digraph, let $p$ be prime, let $k$ be a positive integer and let $w: E(\G) \rightarrow \Z_p^k$. A collection of families of gadgets $\cG_1, \ldots, \cG_t$  in $\G$ is said to be \emph{useful} if it satisfies the following properties:
			\begin{enumerate}[label = \rm(P\arabic*)]
				\item  \label{itm:gadget-disjoint}
					The collections $\cG_1, \ldots, \cG_t$ are pairwise disjoint, and
					$V(g_1) \cap V(g_2) = \emptyset$ for all distinct $g_1, g_2 \in \bigcup_{i=1}^t\cG_i$;
				\item \label{itm:gadget-size}
					$|V(g)| = 3$ for all $g \in \bigcup_{i = 1}^t\cG_i$;
				\item  \label{itm:gadget-lex}
					defining the multiset $U_i := \bigmcup_{g \in \mathcal{G}_i}g^*$, the sequence $|\Sigma(U_1)|, \ldots, |\Sigma(U_t)|$ is as late as possible in the lexicographic ordering\footnote{Recall that $a_1,a_2,\ldots, a_t < b_1,b_2,\ldots, b_t$ in the \emph{lexicographic ordering} on $\mathbb{R}^t$ if there is some $i \in [t]$ such that $a_i < b_i$ and $a_j = b_j$ for all $j < i$.}  on $\mathbb{R}^t$ among sequences satisfying \ref{itm:gadget-disjoint} and \ref{itm:gadget-size}.
			\end{enumerate}
			We say that a useful family is also \emph{reduced} if the following holds.
			\begin{enumerate}[label = \rm(P\arabic*), resume]
				\item \label{itm:gadget-reduced}
					$U_i$ is a reduced multiset, for every $i \in [t]$.
			\end{enumerate}
			By \Cref{obs:reduced}~\ref{itm:cont-reduced}, for every useful family of gadgets $\cG_1, \ldots, \cG_t$ in $\G$ there is a useful and reduced family $\cG_1', \ldots, \cG_t'$ with $\Sigma(U_i') = \Sigma(U_i)$ (where $U_i' := \bigmcup_{g \in \cG_i'}g^*$). 

			For each $i \in [t]$ define
			\begin{equation}\label{eq:VBd}
				V_i := \bigcup_{g \in \cG_i}V(g), \qquad B_i:= \{x \in \mathbb{Z}_p^k: x + \Sigma(U_i) = \Sigma(U_i)\}.
			\end{equation}
		\end{definition}

		Observe that the sets $V_i$ are pairwise disjoint, that $B_i$ is a subgroup of $\Z_p^k$, and $B_i \subseteq \Sigma(U_i)$ (using $0 \in \Sigma(U_i)$). Let $d_i := \dim(B_i)$, where $B_i$ is treated as a vector subspace of the vector space $\Z_p^k$.   

		The next lemma shows that once we have `pulled out' a useful collection of gadgets, we have control over the weights of many of the edges in the digraph.

		\begin{lemma}\label{lem:gadseq}
			Let $\Gamma$ be a complete digraph, let $p$ be prime and let $w: E(\G) \rightarrow \Z_p^k$ be a labelling with no zero-sum cycles. Let $\cG_1, \ldots, \cG_t$ be a useful collection of families of gadgets in $\G$, and define $V_i, B_i, d_i$ as above. Suppose that $v_0 \in V(\G) \setminus \bigcup_{j = 1}^t V_j$ satisfies $w(v_0 u) = 0$ for every vertex $u \neq v_0$. Then:
			\begin{enumerate}[label = \rm(\roman*)]
				\item\label{it:wtsinB}
					For each $i \in [t]$, every edge $e \in \G \setminus (\{v_0\} \cup \bigcup_{j=1}^i V_j)$ satisfies $w(e) \in B_i$.
				\item\label{it:Bdec}
					We have $k > d_1 > \ldots > d_t$.
			\end{enumerate}    
		\end{lemma}

		\begin{proof}

			Suppose that $e = uv$ violates \ref{it:wtsinB} for some $i \in [t]$. Then the gadget $g$ with $V(g) = \{u, v, v_0\}$, $P(g) = v_0 v$ and $Q(g) = v_0 u v$ has value 
			\begin{equation*}
				w(Q(g)) - w(P(g)) = w(v_0uv) - w(v_0v) = w(uv),
			\end{equation*}
			which is not in $B_i$.
			By definition of $B_i$, this shows $|\Sigma(U_i \cup \{g^*\})| > |\Sigma(U_i)|$. Thus, taking $\cG_i' := \cG_i \cup \{g\}$, we reach a contradiction to property \ref{itm:gadget-lex}.

			Define for convenience $B_0 = \Z_p^k$ and $d_0 = k$, and let $i \in [t-1]$.
			For \ref{it:Bdec}, by \ref{it:wtsinB} we have $U_i \subseteq B_{i-1}$ (this holds trivially for $i = 1$). It follows that $B_{i} \subseteq \Sigma(U_{i}) \subseteq B_{i-1}$ (where the first inclusion holds by definition of $B_i$ and the second holds because $B_{i-1}$ is a group), and thus $d_{i} \le d_{i-1}$. 
			Suppose, towards a contradiction, that $d_{i} = d_{i-1}$. Then $B_{i-1} = B_{i}$ and, by $B_{i} \subseteq \Sigma(U_{i}) \subseteq B_{i-1}$, we get $B_{i} = \Sigma(U_{i})$.
			But now applying Lemma~\ref{lem:sigmall} with the subgraph $\Gamma[V_{i}]$, whose edges are weighted by elements of $\Sigma(U_{i})$, yields a zero-sum cycle. This contradicts our assumption that $\G$ contains no zero-sum cycles. Hence $d_{i-1} > d_{i}$, proving \ref{it:Bdec}. 
		\end{proof}

\section{Zero-sum cycles and reduced sets}\label{sec:proof}

	Given a prime $p$ and integer $k$, recall that $f(\Z_p^k)$ is the minimum $f$ such that every complete digraph on $f$ vertices, whose edges have weights in $\Z_p^k$, has a zero-sum cycle. For ease of notation, write $f_p(k):= f(\Z_p^k)$. Recall that $h_p(k)$ is the maximum size of a reduced multiset in $\Z_p^k$.

	The main result of this section is the following lemma, which together with \Cref{thm:reduced-upper} yields \Cref{thm:main}.

	\begin{lemma}\label{lem:step}
		Let $p$ be a prime and let $k \ge 1$ be integer. Then 
		\begin{equation} \label{eqn:step}
			f_p(k) \le 60 \log (2k) \cdot h_p(10k).
		\end{equation}
	\end{lemma}

	Before proving \Cref{lem:step}, we show how it implies \Cref{thm:main} and \Cref{thm:pis2}.

	\begin{proof}[Proof of \Cref{thm:main}]
		By \Cref{lem:step} and \Cref{thm:reduced-upper}, we have
		\begin{align*}
			f_p(k) &\le 60 \log (2k) \cdot h_p(10k) \\
			&\le 60 \log (2k) \cdot (p-1) (\log (10 k) + 1) \cdot (10k)\\
			&\le 600 p \cdot k (\log (10k))^2,
		\end{align*}
		as required. 
	\end{proof}

	\begin{proof}[Proof of \Cref{thm:pis2}]
		By \Cref{obs:reduced}\ref{itm:reduced-dim-1} and \Cref{thm:reduced-upper}, we have
		\begin{align*}
			f_2(k) &\le 60 \log (2k) \cdot h_2(10k) 
			\le 60 \log (2k) \cdot 10k = 600 \cdot k\log(2k).
		\end{align*}
		as required. 
	\end{proof}

	\begin{proof}[Proof of Lemma~\ref{lem:step}]

		We prove the lemma by induction on $k$. Notice that, by \cite{alon2020divisible,MS21}, $f_p(1) \le 2p$, and by \Cref{obs:reduced}\ref{itm:reduced-p-2} we have $60 \cdot h_p(10) \ge 60 \cdot h_p(1) = 60(p-1)$, so the statement holds for $k = 1$.

		Let $\G$ be a complete digraph on $n$ vertices that does not contain a zero-sum cycle. Our goal is to show that $n \le 60 \log(2k) \cdot h_p(10k)$. Suppose to the contrary that $n \ge 60 \log(2k) \cdot h_p(10 k)$. 

		Set $t := \ceil{10 \log k}$.
		Choose $\cG_1, \ldots, \cG_t$ to be a collection of families of gadgets in $\G$ satisfying \ref{itm:gadget-disjoint} and \ref{itm:gadget-size}, such that the sequence $|\Sigma(U_1)|, \ldots, |\Sigma(U_t)|$ is as late as possible in the lexicographic order among all such sequences, where $U_i$ is as in \ref{itm:gadget-lex}, so that \ref{itm:gadget-lex} holds. Notice that such a sequence exists (taking $\cG_i = \emptyset$ yields a valid sequence). By replacing $U_i$ by a minimal subset $U_i' \subseteq U_i$ satisfying $\Sigma(U_i') = \Sigma(U_i)$, we may additionally assume that \ref{itm:gadget-reduced} holds. Define $V_i, B_i, d_i$ as in \eqref{eq:VBd}.

		For each $i \in [t]$, as $U_i$ is reduced, $|U_i| \le h_p(k)$  and thus $|V_i| \le 3h_p(k)$ (by \ref{itm:gadget-size}), showing $|\bigcup_{i \in [t]} V_i| \le 3t \cdot h_p(k) \le n-1$. 
		Let $v_0$ be a vertex not in $\bigcup_{i \in [t]}V_i$. Apply \Cref{lem:zero-vertex} to obtain $w' : E(\G) \to \Z_p^k$ such that: $w'(v_0 u) = 0$ for every vertex $u \neq v_0$; there are no $w'$-zero-sum cycles; and every gadget in $\G$ has the same weight with respect to $w$ and $w'$. In particular, a sequence of gadgets is useful with respect to $w$ if and only if it is useful with respect to $w'$, and similarly for useful and reduced sequences. Altogether, this means that we may and will replace $w$ by $w'$, allowing us to assume $w(v_0u) = 0$ for $u \in V(\G) \setminus \{v_0\}$.

		Hence we can apply Lemma~\ref{lem:gadseq} to obtain:

		\begin{enumerate}[label = \rm(\roman*)]
			\item \label{itm:d-a}
				For each $i \in [t]$, every edge $e$ in $\G \setminus (\bigcup_{j=1}^i V_j \cup \{v_0\})$ satisfies $w(e) \in B_i$.
			\item \label{itm:d-b}
				We have $k > d_1 > \ldots > d_t$. In particular, $0 < d_i < k$ for $i \in [t-1]$.
		\end{enumerate}

		\begin{claim}\label{cl:m1}
			There exists $m \in [3, t]$ such that $k - d_m \le 10(k-d_{m-2})$.
		\end{claim}
		\begin{proof}
			Suppose $k - d_m > 10(k-d_{m-2})$ for all $m \in [3,t]$. 
			Then, as $1 \le k - d_1 \le k - d_2 \le k$,
			$$k \ge k - d_t > 10^{t/2-1}(k - d_2) \ge 10^{t/2-1} \ge k^5/10 > k,$$
			a contradiction.
		\end{proof}

		Fix $m$ as in Claim~\ref{cl:m1}. Our next goal is to show that there is a small set $Z$ of vertices from $\G$ such that the weights of all edges not touching $Z$ lie in a proper subgroup of $\Z_p^k$. We can then apply our inductive hypothesis to bound $|V(\G \setminus Z)|$.

		For each $i \in [t]$, let $\tau_i : \Z_p^k \to \Z_p^k / B_i$ be the natural map sending $a$ to $a + B_i$ for $a \in \Z_p^k$. For a multisubset $X$ of $\Z_p^k$, let $\tau_i(X):= \bigmcup_{x \in X}\tau_i(x)$. 

		For each $j \in [m-2]$, define $\mathcal{X}_j$ to be a minimal subset of $\mathcal{G}_j$ such that $X_j:= \bigmcup_{g \in \mathcal{X}_j}g^*$ satisfies $\Sigma(\tau_m(X_j)) = \Sigma(\tau_m(U_j))$. Similarly, define $\mathcal{Y}_j$ to be a minimal set of gadgets in $\mathcal{G}_j$ such that $Y_j:= \bigmcup_{g \in \mathcal{Y}_j}g^*$ satisfies $\Sigma(\tau_{m-1}(Y_j)) = \Sigma(\tau_{m-1}(U_j))$. 
		Observe that $\tau_m(X_j)$ is reduced in $\Z_p^k/B_m$ and $\tau_{m-1}(Y_j)$ is reduced in $\Z_p^k / B_{m-1}$.
		In particular, 
		\begin{equation} \label{eqn:reduced}
			|\cX_j| \le h_p(k - d_m), \qquad |\cY_j| \le h_p(k - d_{m-1}).
		\end{equation}
		Notice that, by definition of $\cX_j$ and $\cY_j$,
		\begin{equation} \label{eqn:X-Y}
			\Sigma(U_j) \subseteq \Sigma(X_j \mcup U_m), \qquad
			\Sigma(U_j) \subseteq \Sigma(Y_j \mcup U_{m-1}).
		\end{equation}
		Indeed, if $u \in \Sigma(U_j)$ then there exist $x \in \Sigma(X_j)$ and $v \in B_m$ such that $u = x + v$, showing $\Sigma(U_j) \subseteq \Sigma(X_j) + B_m \subseteq \Sigma(X_j) + \Sigma(U_m) = \Sigma(X_j \mcup U_m)$. A similar argument holds for $Y_j$. 
		Define
		\begin{equation} \label{eqn:Z}
			Z := \{v_0\} \cup \bigcup_{i \in [m-2]} \big(V(\cX_i) \cup V(\cY_i)\big),
		\end{equation}
		where $V(\cX_i)$ denotes the union of vertex sets of gadgets in $\cX_i$, and similarly for $V(\cY_i)$.

		\begin{claim}\label{cl:z-size}
			$|Z| \le 6m \cdot h_p(k - d_m)$.
		\end{claim}
		\begin{proof}
			We have
			\begin{align*}
				|Z| 
				& \le 3\sum_{j \in [m-2]} \big(|\cX_j| + |\cY_j|\big) + 1 \\
				& \le 3\sum_{j \in [m-2]}\big(h_p(k - d_m) + h_p(k - d_{m-1})\big) + 1 \\
				& \le 6 m \cdot h_p(k - d_m),
			\end{align*}
			where for the first inequality we used that gadgets consist of three vertices, for the second we used \eqref{eqn:reduced}, and for the third we used \ref{itm:d-b} and the monotonicity of $h_p(\cdot)$.
		\end{proof}

		\begin{claim}\label{cl:Z-edges}
			Every edge $e \in E(\G \setminus Z)$ satisfies $w(e) \in B_{m-2}$.
		\end{claim}

		\def \Vi {V_{\infty}}
		\begin{proof}
			Suppose there exists an edge $xy$ with no endpoint in $Z$ such that $w(xy) \notin B_{m-2}$. 
			Define, for convenience, $\Vi := V(\G) \setminus \bigcup_{i \in [m]} V_i$. Let $g$ be the gadget rooted at $(v_0,y)$ with $V(g) = \{v_0,x,y\}$, $P(g) = v_0y$, and $Q(g) = v_0 xy$. Its value is
			\begin{equation*}
				g^* = w(Q(g)) - w(P(g)) 
				= w(v_0xy) - w(v_0y)
				= w(v_0x) + w(xy) - w(v_0y) = w(xy),
			\end{equation*}
			using $w(v_0u) = 0$ for $u \neq v_0$. As $w(xy) \notin B_{m-2}$, it follows that $g^* \notin B_{m-2}$.

			We will show that it is possible to find a collection of families of vertex-disjoint gadgets $\cG'_1, \ldots, \cG'_{m-3}, \cG''_{m-2}$ in $\G$ with corresponding sets $U'_1,\ldots,U_{m-3}',U''_{m-2}$ (where $U'_i := \bigmcup_{g \in \cG'_i}g^*$ and similarly for $U''_{m-2}$) satisfying properties \ref{itm:gadget-disjoint} and \ref{itm:gadget-size} with the additional conditions
			\begin{enumerate}[label = \rm(C\arabic*)]
				\item \label{itm:lex-equal}
					$|\Sigma(U'_j)| \ge |\Sigma(U_j)|$ for all $j \in [m-3]$, and 
				\item \label{itm:lex-larger}
					$|\Sigma(U''_{m-2})| > |\Sigma(U_{m-2})|$. 
			\end{enumerate}
			This will imply that $\cG_1, \ldots, \cG_t$ does not satisfy property \ref{itm:gadget-lex}, as we are able to obtain a sequence that is later in the lexicographic ordering, a contradiction.

			Suppose $x \in V_a$ and $y \in V_b$, where $a, b \in [m] \cup \{\infty\}$ and $a \le b$. We need an ad-hoc piece of notation; let $m', m''$ satisfy $\{m', m''\} = \{m-1, m\}$ and, if $b \in \{m-1, m\}$ but $a \notin \{m-1, m\}$, then $m' \neq b$.
			We define $\cG'_1, \ldots, \cG'_{m-2}$ as follows. 
			$$\cG'_j := 
			\begin{cases}
				\cG_j & \text{ if $j \in [m-2] \setminus \{a,b\}$,} \\
				\mathcal{X}_j \cup \cY_j \cup \cG_{m'} & \text{ if $j \in [m-2]$ and $j = a$} \\
				\cX_j \cup \mathcal{Y}_j \cup \cG_{m''} & \text{ if $j \in [m-2]$ and $j = b > a$.}
			\end{cases}
			$$
			Notice that the gadgets in each $\cG_j$, with $j \in [m-2]$, are pairwise vertex-disjoint. Indeed, this holds because $V(\cX_j), V(\cY_j) \subseteq V_j$ and the sets $V_j$ are pairwise vertex-disjoint. 
			Similarly, any two gadgets from distinct sets $\cG_j'$ are vertex-disjoint. 
			Additionally, the gadgets in $\bigcup_{j \in [m-2]} \cG_j'$ are vertex-disjoint from $(V_a \cup V_b \cup V_{\infty}) \setminus Z$ (if $b \in \{m-1, m\}$ and $a \notin \{m-1, m\}$ this follows from the choice of $m'$), and hence they are vertex-disjoint from $V(g) = \{x,y,v_0\}$. Altogether, writing $\cG_{m-2}'' := \cG_{m-2}' \cup \{g\}$, this shows that $\cG_1', \ldots, \cG_{m-3}', \cG_{m-2}''$ satisfies \ref{itm:gadget-disjoint}. Property \ref{itm:gadget-size} holds by construction.

			Next, we claim that $\Sigma(U_j') = \Sigma(U_j)$ for $j \in [m-2]$, where $U'_{m-2} := \bigmcup_{g \in \cG'_{m-2}}g^*$. This is clear when $j \neq a, b$. Suppose now that $\cG_j' = \cX_j \cup \cY_j \cup \cG_{m-1}$.  If $u \in \Sigma(U_j)$, then by choice of $\cY_j$ there exists $w \in \Sigma(Y_j)$ such that $u + B_{m-1} = w + B_{m-1}$. Equivalently, $u - w \in B_{m-1}$ and thus $u - w \in \Sigma(U_{m-1})$. This implies that $u \in \Sigma(Y_j \mcup U_{m-1}) \subseteq \Sigma(U_j')$, as claimed. An analogous argument shows the same when $\cG_j' = \cX_j \cup \cY_j \cup \cG_m$. Hence $|\Sigma(U'_j)| \ge |\Sigma(U_j)|$ for all $j \in [m-2]$, which proves \ref{itm:lex-equal}.

			In particular, $\Sigma(U_{m-2}) \subseteq \Sigma(U_{m-2}')$. This, the definition of $B_{m-2}$, and that $g^* \notin B_{m-2}$ show that
			\begin{equation*}
				\Sigma(U_{m-2}'') = \Sigma(U_{m-2}' \mcup \{g^*\})  
				\supseteq \Sigma(U_{m-2} \mcup \{g^*\}) \supsetneq \Sigma(U_{m-2}).
			\end{equation*}
			Thus \ref{itm:lex-larger} holds.
		\end{proof}

		Claims~\ref{cl:z-size} and \ref{cl:Z-edges} show that, by removing at most $6m \cdot h_p(k - d_m)$ vertices from $\G$, we can obtain a graph whose edge weights are contained in a subgroup isomorphic to $\Z_p^{d_{m-2}}$. But as we presumed $\G$ has no zero-sum cycle, then $|V(\G \setminus Z)| \le f_p(d_{m-2})$. Hence
		\begin{align*}
			n = |V(\G)| &= |V(Z)| + |V(\G \setminus Z)|\\
			& \le 6m \cdot h_p(k - d_m) + f_p(d_{m-2}) \\
			& \le 60 \log k \cdot h_p\big(10(k - d_{m-2})\big) + 60 \log(2d_{m-2}) \cdot h_p(10d_{m-2}) \\
			& \le 60 \log (2k) \cdot h_p\big(10(k - d_{m-2})\big) + 60 \log(2k) \cdot h_p(10d_{m-2}) \\
			& \le 60 \log (2k) \cdot h_p(10k).
		\end{align*}
		where in the third line we used the bounds $m \le 10 \log k$ and $k - d_m \le 10(k - d_{m-2})$ and the monotonicity of $h_p(\cdot)$, as well as the induction hypothesis on $f_p(\cdot)$ (using $d_{m-2} \le k-1$; see \ref{itm:d-b}), and in the last line we applied \Cref{obs:reduced}\ref{itm:reduced-sub-additive}. 
	\end{proof}

\section{A few concluding remarks}\label{sec:conc}

	Our main result shows $f(\mathbb{Z}_p^k) = O(pk (\log k)^2)$ when $p$ is prime, with a better bound of $O(k \log k)$ when $p = 2$. This is tight up to a factor which is polylogarithmic in $k$, due to the easy lower bound $f(\Z_p^k) \ge (p-1)k$. It would be nice to close the gap between these upper and lower bound. 
	\begin{question}
		Is it true that $f(\Z_p^k) = O(pk)$?
	\end{question}
	It would also be interesting to determine whether similar bounds hold when $p$ is not prime. 

	\Cref{lem:step} provides a bound on $f_p(k)$ in terms of $h_p(k)$. In light of this, one way to improve \Cref{thm:main} could be to improve our understanding of $h_p(k)$. Indeed, proving that $h_p(k) = O(pk)$ would give an improvement of a $\log(k)$ factor in \Cref{thm:main}. Given that we know (see \Cref{obs:reduced}~\ref{itm:reduced-dim-1}) that $h_p(1)=p-1$, a next step towards understanding $h_p(k)$ could be to determine $h_p(2)$. \Cref{thm:reduced-upper} gives $h_p(2) \le 4(p-1)$. To this end, we can prove the following.

	\begin{restatable}{lemma}{lemReducedUpper} \label{lem:reduced-upper-2}
		Let $p \ge 7$ be prime. Then $h_p(2) < \frac{5}{2}(p-1)$.
	\end{restatable}

	For completeness, we provide a proof in \Cref{sec:appendix-a}. However, we do not think that this bound is best possible. 

	\begin{conjecture}
		 Let $p$ be prime. Then there exists a constant $C$ such that $h_p(2) \le 2p + C$.   
	\end{conjecture}

    We believe that the proof methods in this paper could be extended to give bounds on $f(G)$ for other finitely generated abelian groups $G$. It is also a natural interesting question in its own right to consider the problem of bounding the size of a largest reduced multiset in such a $G$. 

	\subsection*{Acknowledgements}

		We are very grateful to Lisa Sauermann for coming up with \Cref{thm:reduced-upper} and for sharing her solution with us; this allowed us to both shorten our paper and improve our results. We would also like to thank Noga Alon for interesting discussions regarding reduced sets and for bringing \cite{alon1991additive} to our attention.
		These discussions took place during the workshop `Recent advances in probabilistic and extremal combinatorics' which was held in Ascona, Switzerland on 7--12 August 2022. The first author would like to thank Asaf Shapira and Benny Sudakov, the workshop's organisers, for inviting her and giving her the opportunity to give a talk. We would also like to thank the referees for their helpful comments.

\bibliography{zero-cycle}
\bibliographystyle{amsplain}

\appendix

\section{Proof of Lemma~\ref{lem:reduced-upper-2}}\label{sec:appendix-a}

	The next lemma is a rephrasing of Corollary 3.4 in \cite{alon1991additive}. We give a similar proof here.

	\begin{lemma}[Corollary 3.4 in \cite{alon1991additive}, rephrased] \label{lem:polynomial}
		Let $v_i = (v_i(1), \ldots, v_i(k))$ be a sequence of $k(p-1)$ vectors in $\Z_p^k$.
		Suppose that 
		\begin{equation}\label{eq:coeff-cond}
			\sum_{(I_1, \ldots, I_k) \in \cI} \,\, \prod_{i \in [k]} \prod_{j \in I_i} v_j(i) \neq 0,
		\end{equation}
		where $\cI$ is the collection of equipartitions of $[k(p-1)]$. 
		Then $\Sigma(\{v_1, \ldots, v_{k(p-1)}\}) = \Z_p^k$.
	\end{lemma}

    	Our proof uses Alon's Combinatorial Nullstellensatz.

	\begin{theorem}[Alon's Combinatorial Nullstellensatz~\cite{alonnul}]\label{thm:alon}
		Let $f \in \mathbb{F}[x_1,\ldots,x_n]$ be a polynomial with coefficients in a field $\mathbb{F}$ such that the degree of $f$ is $\sum_{i=1}^n t_i$ and the coefficient of $\prod_{i=1}^n x_i^{t_i}$ is non-zero. Let $S_1,\ldots,S_n$ be subsets of $\mathbb{F}$ such that $|S_i| > t_i$ for all $i$. Then there exists  $(s_1,\ldots,s_n) \in S_1 \times \ldots \times S_n$ such that $f(s_1,\ldots,s_n) \not= 0$. 
	\end{theorem}

	\begin{proof}[Proof of \Cref{lem:polynomial}]
		Fix $w = (w(1), \ldots, w(k)) \in \Z_p^k$. Set $m := k(p-1)$ and define
		\begin{equation} \label{eqn:polynomial}
			P(x_1,\ldots,x_m) := 
			\prod_{i=1}^{k}\left(\bigg(\sum_{j \in [m]}x_jv_j(i) - w(i)\bigg)^{p-1} - 1\right).
		\end{equation}
		Let $S$ be the multiset $\bigmcup_{i \in [m]}v_i$.
		Observe that $P(y) \neq 0$ for some $y \in \{0,1\}^m$ if and only if $w \in \Sigma(S)$ (via Fermat's little theorem). 

		Thus, applying Theorem~\ref{thm:alon}, with $S_i = \{0,1\}$ and $t_i = 1$ for every $i \in [m]$, we see that if the coefficient of $\prod_{i \in [m]}x_i$ in $P$ is non-zero then $w \in \Sigma(S)$. This coefficient is the coefficient of $\prod_{i \in [m]}x_i$ in $\prod_{i=1}^{k}\left(\sum_{j \in [m]} x_jv_j(i) \right)^{p-1}$. In order to obtain a term of $\prod_{i \in [m]}x_i$ from the latter product, from each factor we must select a distinct variable $x_j$ in such a way that every $j \in [m]$ appears exactly once. 
		Recalling that $\cI$ is the collection of equipartitions $(I_1, \ldots, I_k)$ of $[m]$ and thinking of $I_i$ as indexing the variables obtained from the $i$th factor of the product, this coefficient is
		\begin{equation*}
			\big((p-1)!\big)^k\sum_{(I_1,\ldots,I_k) \in \mathcal{I}}\,\, \prod_{i \in [k]} \,\prod_{j \in I_i}v_j(i).
		\end{equation*}
		As $p$ is prime, $(p-1)! \neq 0$, and \eqref{eq:coeff-cond} holds if and only if the coefficient of $\prod_{i \in [m]}x_i$ in $P$ is non-zero, which implies that $w \in \Sigma(S)$. Since $w$ was an arbitrary element in $\Z_p^k$, this shows that \eqref{eqn:polynomial} implies $\Sigma(S) = \Z_p^k$.
	\end{proof}

	We now prove \Cref{lem:reduced-upper-2}, restated here.
	\lemReducedUpper*

	\begin{proof}[Proof of \Cref{lem:reduced-upper-2}]
		Suppose that $T \msubseteq \mathbb{Z}_p^2$ is reduced and $|T| \ge \frac{5}{2}(p-1)$. Let $S \subseteq T$ satisfy $|S| = \frac{5}{2}(p-1)$.  By \Cref{obs:reduced}~\ref{itm:reduced-subset}, $S$ is reduced. By  \Cref{obs:reduced}~\ref{itm:identity}, $(0,0) \notin S$. We will show that $\Sigma(S) = \Z_p^2$, which is a contradiction as this implies $T$ is not reduced. 
		
		For $v \in \mathbb{Z}_p^2 \setminus \{(0,0)\}$, define $\angle{v} := \{\alpha v : \alpha \in \Z_p\}$. 
		Choose $v$ to minimise the size of the multiset intersection $\angle{v} \mcap S$, defined to be the multisubset of $S$ whose elements lie in $\angle{v}$. 
		Notice that there are $p+1$ different `directions', namely $(1, i)$ for $i \in [0,p-1]$ and $(0, 1)$, and each $v \in \mathbb{Z}_p^k$ lies in $\angle{d}$ for exactly one direction $d$. 
		Hence, as $|S| = \frac{5}{2}(p-1) < 3(p+1)$, we have $|\angle{v} \mcap S| \le 2$. 
		Let $f : \Z_p^2 \to \Z_p^2$ be an injective linear transformation mapping $v$ to $(0,1)$. Then, by \Cref{obs:reduced}~\ref{itm:reduced-map}, $f(S)$ is a reduced multisubset of $\Z_p^2$ of size $\frac{5}{2}(p-1)$ with the property that $f(S)$ contains at most two vectors in direction $(0,1)$. That is, the multiset intersection $f(S) \mcap \angle{(0,1)}$ has size at most 2.

		Let $S_0$ be obtained from $S$ by removing all elements in $\angle{(0,1)}$, and observe that $|S_0| \ge \frac{5}{2}(p-1) - 2 > p-1$. Let $v_1, \ldots, v_{p-1}$ be chosen as follows,
		defining $S_i := S_0 - \{v_1, \ldots, v_i\}$ once $v_1, \ldots, v_i$ are defined. 
		Having chosen $v_1, \ldots, v_{i-1}$, choose $v_i \in S_{i-1}$ to maximise $|\angle{v_i} \mcap S_{i-1}|$.

		Write $S' := S - \{v_1, \ldots, v_{p-1}\}$ and define
		\begin{equation} \label{eqn:max}
			m := \max_{v \in \Z_p^2 \setminus \{(0,0)\}} |S' \mcap \angle{v}|.
		\end{equation}
		\begin{claim}\label{cl:m-bound}
			$m \le (p-1)/2$.
		\end{claim}
		\begin{proof}
			Let $T$ be the set of elements $t \in \Z_p$ such that $\angle{(1,t)} \cap \{v_1, \ldots, v_{p-1}\} \neq \emptyset$. Then, by definition of $v_1, \ldots, v_{p-1}$, we have $|S' \mcap \angle{(1,t)}| \in \{m-1, m\}$ for every $t \in T$. 

			If $|T| \le 2$ then 
			\begin{equation*}
				2(p-1) \ge |S \mcap \bigcup_{t \in T} \angle{(1,t)}| \ge p-1 + m + m - 1,
			\end{equation*}
			using that $|S \mcap \angle{v}| \le p-1$ for $v \in \Z_p^2 \setminus \{(0,0)\}$ (which follows from \Cref{obs:reduced}~\ref{itm:reduced-dim-1}). This implies that $m \le p/2$, and thus $m \le (p-1)/2$ because $p$ is odd.

			If $|T| \ge 3$ then 
			\begin{equation*}
				\frac{5}{2}(p-1) \ge \left|S \cap \bigmcup_{t \in T} \angle{(1,t)}\right| \ge p-1 + m + (|T|-1)(m-1) \ge p + 3m - 3.
			\end{equation*}
			Thus $m \le p/2 + 1/6$, implying that $m \le (p-1)/2$, as $p$ is odd. This completes the proof of Claim~\ref{cl:m-bound}. 
		\end{proof}

		\begin{claim}\label{cl:non-zero}
			For every $t \in [0, p-1]$ there exists a multisubset $\bigmcup_{i \in [p-1+t]} v_i \subseteq S'$ such that
			\begin{equation} \label{eqn:polynomial-dim-2}
				\sum_{(I_1, I_2) \in \cI_t} \,\, \prod_{i \in [2]} \prod_{j \in I_i} v_j(i) \neq 0,
			\end{equation}
			where $\cI_t$ is the collection of partitions $(I_1, I_2)$ of $[p-1+t]$ such that $|I_1| = p-1$.
		\end{claim}

		\begin{proof}
			We prove the claim by induction on $t$. When $t = 0$, the expression in \eqref{eqn:polynomial-dim-2} is $\prod_{j \in I_1} v_j(1)$, which is non-zero as $v_1, \ldots, v_{p-1} \notin \angle{(0,1)}$.
			 
			Now suppose that $t \in [p-1]$ and $v_1, \ldots, v_{p-1+t-1}$ satisfy the requirements of the claim for $t-1$. Expanding the left-hand side of \eqref{eqn:polynomial-dim-2}, with $v_{p-1+t} = (x,y)$, gives
			\begin{equation*}
				x s_1 + y s_2,
			\end{equation*}
			where $s_1$ is a sum of terms depending on $v_1, \ldots, v_{p-1+t-1}$ and 
			\begin{equation*}
				s_2 = \sum_{(I_1, I_2) \in \cI_{t-1}} \,\, \prod_{i \in [2]} \prod_{j \in I_i} v_j(i) \neq 0.
			\end{equation*}
			The multisubset of $S \mminus \{v_1, \ldots, v_{p-1+t-1}\}$ of vectors $(x,y)$ satisfying $xs_1 + ys_2 = 0$ is contained in a subspace of dimension $1$. Because $m \le (p-1)/2$ (by Claim~\ref{cl:m-bound}) and $|S \mminus \{v_1, \ldots, v_{p-1+t-1}\}| \ge (p+1)/2$, we find that there is a suitable $v_{p-1+t} \in S \mminus \{v_1, \ldots, v_{p-1+t-1}\}$.
		\end{proof}
		
		\Cref{lem:polynomial} and \Cref{cl:non-zero} (with $t = p-1$) imply that there is a multisubset $S' \subseteq S$ of size $2(p-1)$ such that $\Sigma(S') = \Z_p^2$. As $S' \not= S$, this contradicts the assumption that $S$ is reduced. 
	\end{proof}

\end{document}